\documentclass[12pt]{amsart}
\usepackage{amsmath,amssymb,amsbsy,amsfonts,amsthm,latexsym,
                        amsopn,amstext,amsxtra,euscript,amscd,mathrsfs,color,bm,cite}
\usepackage{float}
\usepackage[english]{babel}
\usepackage{mathtools}
\usepackage{todonotes}
\usepackage{url}
\usepackage[colorlinks,linkcolor=blue,anchorcolor=blue,citecolor=blue,backref=page]{hyperref}
\usepackage{cases}
\usepackage{txfonts}
\usepackage{verbatim}

\def\r{\right}

\newtheorem{theorem}{Theorem}

\newtheorem{lemma}[theorem]{Lemma}
\newtheorem{conjecture}[theorem]{Conjecture}
\numberwithin{equation}{section}
\numberwithin{theorem}{section}
\newtheorem*{Remark}{Remark}

\newtheorem*{definition}{Definition}
\newtheorem*{question}{Question}
\def\qed{\ifhmode\textqed\fi
   \ifmmode\ifinner\quad\qedsymbol\else\dispqed\fi\fi}
\def\textqed{\unskip\nobreak\penalty50
    \hskip2em\hbox{}\nobreak\hfil\qedsymbol
    \parfillskip=0pt \finalhyphendemerits=0}
\def\dispqed{\rlap{\qquad\qedsymbol}}

\begin{document}
\title{O\MakeLowercase{n the set of} K\MakeLowercase{ronecker numbers}}
\author{Sayan Goswami}
\address{The Institute of Mathematical Sciences, A CI of Homi Bhabha National Institute, CIT Campus, Taramani, Chennai 600113, India.}
\email{sayan92m@gmail.com}

\author{Wen Huang}
\address {School of Mathematical Sciences, University of Science and Technology of China, Hefei 230026, P. R.
China.}
\email {wenh@mail.ustc.edu.cn}

\author{XiaoSheng Wu}
\address {School of Mathematics, Hefei University of Technology, Hefei 230009,
P. R. China}
\email {xswu@amss.ac.cn}
\date{}

\begin{abstract}
An positive even number is said to be a Kronecker number if it can be written in infinitely many ways as the difference between two primes, and it is believed that all even numbers are Kronecker numbers. 
We will study the division and multiplication of Kronecker numbers to study the largeness of the set of Kronecker numbers. A numerical lower bound for the density Kronecker numbers among even numbers is given, and it is proved that there exists a computable constant $k$ and a set $D$ consisting of at most 720 computable Maillet numbers such that, for any integer $n$, $kn$ can be expressed as a product of a Kronecker number with a Maillet number in $D$. Meanwhile, it is proved that every positive rational number can be written as a ratio of two Kronecker numbers.

\end{abstract}
\maketitle
\section{Introduction}

The distribution of the difference of primes is a widely concerned theme in number theory, and the following conjecture is well-known.
\begin{conjecture}[Kronecker \cite{Kro01}]
Every even number can be written in infinitely many ways as the difference of two primes.
\end{conjecture}
The conjecture is beyond the reach of modern mathematics. However, series of recent breakthroughs on the twin prime conjecture indicate that there is an integer not exceeding than 246, which can be written in infinitely many ways as the difference of two primes; see also \cite{GPY09}, \cite{Zha14}, \cite{May15}, and \cite{Pol14}.

\begin{definition}[Maillet number, Kronecker number]
An even number $n$ is called a Maillet number (Kronecker number), if it can be written (in infinitely many ways) as the difference of two primes.
\end{definition}
We apply $\mathcal{K}$ for the set of all Kronecker numbers. The twin prime conjecture is about the lower bound of $\mathcal{K}$.
Another important aspect of the Kronecker conjecture is how ``large'' the set $\mathcal{K}$ is. It is proved by Pintz \cite{Pin16} that $\mathcal{K}$ is a syndetic set, which was also obtained by Granville, Kane, Koukoulopoulos and Lemke Oliver \cite{GKKL14}, using a different method later.

One kind of sets larger than the syndetic set are known as the $\Delta_r^*$-set.
\begin{definition}[$\Delta_r^*$-set]
Let $S$ be a non-empty subset of $\mathbb{N}$, and we define its difference set $\Delta(S)$ via
\[
\Delta(S)=(S-S)\cap\mathbb{N}=\{a-b: a, b\in S, a>b\}.
\]
A subset $A$ of $\mathbb{N}$ is called a $\Delta_r^*$-set if for every subset $S$ of $\mathbb{N}$ with $|S|\ge r$, there is
\[
A\cap\Delta(S)\neq\emptyset.
\]
\end{definition}

In \cite{HW17}, Huang and Wu have proved that $\mathcal{K}$ should be a $\Delta_r^*$-set.
\begin{theorem}\label{thmHW}
The set $\mathcal{K}$ of Kronecker numbers is a $\Delta_r^*$-set for any $r\ge721$.
\end{theorem}

It is also mentioned in \cite{HW17} that the number $721$ can be sharpen to $19$ if the primes have level of distribution $\theta$ for every $\theta<1$. Here, for some given $\theta>0$, we say the primes have `level of distribution $\theta$' if, for any $W>0$, it holds that
\[
   \sum_{q\le x^\theta}\max_{(a,q)=1}\Big|\pi(x;q,a)-\frac{\pi(x)}{\phi(q)}\Big|\ll_W\frac{x}{(\log x)^W}.
\]

In this work, we try to obtain more information on how large $\mathcal{K}$ is, and our deduction starts from Theorem \ref{thmHW}. Our first result is a numerical lower bound for the density of $\mathcal{K}$ among even numbers.
\begin{theorem}\label{thmdensity}
Let $\alpha$ be the density of Kronecker numbers among all positive even numbers. We have $\alpha\ge\frac 1{360}\prod_{p\le 720}(1-p^{-1})$. If the primes have level of distribution $\theta$ for every $\theta<1$, we have $\alpha\ge\frac{1024}{51051}$.
\end{theorem}


Our next result is about the representation of integers by products of differences among primes, which is motivated by the following question, asked by Fish in \cite{Fis18}.
\begin{question}
For a given infinite set $E\subset\mathbb{Z}$, how much structure does the set $(E-E)\cdot(E-E)$ possess ?
\end{question}

Fish \cite{Fis18} considered the question with $E$ being positive density subset of $\mathbb{Z}$. Using Furstenberg's correspondence principle, he proved that there exist $k_0$ (depending on the densities of $E_1$ and $E_2$) and $k\le k_0$ that
\[
k\mathbb{Z}\subset(E_1-E_1)\cdot(E_2-E_2).
\]
It is naturally to consider the question with $E$ being the set of primes, and this case is not contained by Fish's work since the set of primes is an infinite set of $\mathbb{Z}$ but not owning a positive upper Banach density. Based on Theorem \ref{thmHW}, Goswami\cite{Gos} extends Fish's result to the case about primes that
\begin{align}\label{kzp}
k\mathbb{Z}\subset(\mathbb{P}-\mathbb{P})\cdot(\mathbb{P}-\mathbb{P}).
\end{align}

We can say something more for the case about primes. In detail, on the right-hand side of \eqref{kzp}, one factor can be restricted to a finite subset of $\mathbb{P}-\mathbb{P}$, consisting of 720 Maillet numbers, and the other factor takes value among Kronecker numbers. We present our main result in the following.

\begin{theorem}\label{themain}
There exist a computable constant $k$ and a set $D$, consisting of at most 720 computable Maillet numbers, that
\[
k\mathbb{Z}\subset D\cdot \mathcal{K}.
\]
\end{theorem}
The proof of Theorem \ref{themain} would be based on Theorem \ref{thmHW} as well as a recent work on linear equations in primes by Green, Tao, and Ziegler. The number $720$ here would also be sharpen to $18$ if the primes have level of distribution $\theta$ for every $\theta<1$.

To see how large $\mathcal{K}$ is, we also consider the ratio of two Kronecker numbers. The following theorem is about the division of the set $\mathcal{K}.$
\begin{theorem} \label{ratio}
Every positive rational number can be written as a ratio of two elements from $\mathcal{K}.$
\end{theorem}

\section{The density of Kronecker numbers}
In this section, we will deduce a lower bound for the density of a general $\Delta_r^*$-set. As a corollary,  we will prove Theorem \ref{thmHW}.
\begin{lemma}\label{lemaK}
Let $H$ be a $\Delta_r^*$-set, and denote by
\[
A_r(a)=\{a,2a,\dots,(r-1)a\}.
\]
Then, for any integer $a>0$, 
\[
H\cap A_r(a)\neq\emptyset.
\]
\end{lemma}
\begin{proof}
Just taking $S=\{a,2a,\dots,r a\}$, one has
\[
\Delta(S)=\{a,2a,\dots,(r-1)a\}.
\]
Thus, the lemma follows immediately from the definition of $\Delta_r^*$-set.
\end{proof}

\begin{theorem}\label{thmdD}
If $H$ is a $\Delta_r^*$-set, we have
\[
\frac{H\cap[1,N]}{N}\ge \prod_{p\le r-1}\left(1-p^{-1}\right)+o(1).
\]
\end{theorem}
\begin{proof}
By Lemma \ref{lemaK}, every $A_r(a)$ contains at least one element of $H$.
We may obtain a lower bound for the cardinality of $H$ by counting the number of the sets $A_r(a)$, which are disjoint with each other. If $a<b$ be two integers with $A_r(a)\cap A_r(b)\neq\emptyset$, then there are two integers $i,j$, with $(i,j)=1$ and $1\le i<j\le r-1$, satisfying
\[
\frac ab=\frac ij.
\]
Thus, $A_r(a)$ should disjoint with each other if we only counter over such $a$, which does not contain any prime less than $r-1$.
For sufficiently larger $N$, it is obvious that
\[
\left|\left\{a: \Bigg(a,\prod_{p\le r-1}p\Bigg)=1, A_r(a)\subset [1,N]\right\}\right|=\frac N{r}\prod_{p\le r-1}(1-p^{-1})+O(1).
\]
Now the theorem follows immediately.
\end{proof}

\begin{proof}[Proof of Theorem \ref{thmdensity}]
Theorem \ref{thmdensity} follows immediately from Theorems \ref{thmHW} and \ref{thmdD} by taking $r=721$ and $r=19$ respectively.
\end{proof}

\section{Representation of integers}
In this section, we give the proof of Theorem \ref{themain}. In the deduction of the constant $k$ as well as the set $D$, we appeal to
the recent work on linear equations in primes by Green, Tao, and Ziegler.
\subsection{Linear equations in primes}
We present some necessary detail in this section, and more contents about the materials could be found in \cite{GT10}.
Let $d,t$ be integers. A system of affine-linear forms on $\mathbb{Z}^d$ is a collection $\Psi=\{\psi_1,\dots,\psi_t\}$ with $\psi_i: \mathbb{Z}^d\rightarrow \mathbb{Z}$ being affine-linear forms. If $N>0$, the \emph{size} $\parallel \Psi\parallel_N$ of $\Psi$ relative to the scale $N$ is the quantity
\[
\parallel\Psi\parallel_N:=\sum_{i=1}^t\sum_{j=1}^d|\dot{\psi}_i(e_j)|+\sum_{i=1}^t\left|\frac{\psi_i(0)}{N}\right|,
\]
where
\[
\dot{\psi}_i(e_j)=\psi_i(e_j)-\psi_i(0)
\]
with $e_1, e_2, \dots, e_d$ being the standard basis for $\mathbb{Z}^d$. For a system $\Psi$, its \emph{local factor} $\beta_p$ for a prime $p$ is defined via
\begin{align}\label{eqbetap}
\beta_p:=\frac1{p^d}\sum_{n\in\mathbb{Z}_p^d}\prod_{i=1}^t\Lambda_{\mathbb{Z}_p}(\psi_i(n)),
\end{align}
where $\mathbb{Z}_p=\{0,1,\dots,p-1\}$ is the residue class of integers modulo $p$, and $\Lambda_{\mathbb{Z}_p}(n)$ is the \emph{local von Mangoldt function} defined by
\begin{equation*}
	\Lambda_{\mathbb{Z}_p}(n)=\begin{cases}
		\frac{p}{p-1}, & if\ (n,p)=1;\\
		0, & otherwise.
	\end{cases}
\end{equation*}

\begin{definition}[Complexity]
 The complexity of the $\Psi$ is the least integer $s$ that, for each $\psi_i$, one can cover the $t-1$ forms $\{\psi_j:j\neq i\}$ by $s+1$ classes, such that $\psi_i$ does not lie in the affine-linear span of any of these classes; if no such $s$ exists, we say that the complexity is $\infty$.
\end{definition}

Main results of \cite{GT12} and \cite{GTZ12} indicate the following theorem.\begin{theorem}[Green, Tao, and Ziegler]\label{thmGT}
Let $N, d, t, L$ be positive integers, and let $\Psi=\{\psi_1, \dots, \psi_t\}$ be a system of affine-linear forms with size $\parallel \Psi\parallel_N\le L$. Let $K\subset[-N,N]^d$ be a convex body. If $\Psi$ is finite complexity, we have
\[
\#\left\{n\in K\cap\mathbb{Z}^d: \psi_1(n), \dots, \psi_t(n) \ \ \text{prime}\right\}=\left(1+o_{t,d,L}(1)\r)\frac{\beta_\infty}{\log^tN}\prod_{p}\beta_{p}+o_{t,d,L}\left(\frac{N^d}{\log^tN}\r),
\]
where $\beta_\infty:=vol_d\left(K\cap\Psi^{-1}(\mathbb{R}^+)^t\r)$ is typically of size $N^d$, and where the singular product $\prod_{p}\beta_{p}$ is always convergent.
\end{theorem}

It has been pointed out in \cite{GT10}that the singular product $\prod_{p}\beta_{p}$ is always convergent, but it may still vanish since $\beta_p=0$ is possible for small $p$. So, we should note that the conjecture works only for the system of affine-linear forms with $\beta_p\neq 0$, for all $p$ (essentially, for small $p=O_{t,d,L}(1)$ is enough).

For the case with complexity $s>2$, the theorem was first proved under an assumption that the inverse Gowers-norm conjecture and the M\"{o}bius and nilsequences conjecture are true. And then these two conjectures were proved to be true by them in combination with Ziegler; see also \cite{GT12} and \cite{GTZ12}.

\subsection {Proof of Theorem \ref{themain}}
To prove our theorem, we will appeal to a special case of Theorem \ref{thmGT}, which we provide in the following lemma.
\begin{lemma}\label{lem}
Let $\Psi=\{\psi_1,\dots,\psi_t\}$ be a system of affine-linear forms of finite complexity with $\psi_i: \mathbb{Z^+}^d\rightarrow \mathbb{Z^+}$, $1\le i\le t$, and $\beta_p\neq 0$ for any prime $p$. There are infinitely many lattice points $n\in {\mathbb{Z}^+}^d$, which make all $\psi_i(n)$ prime.
\end{lemma}
\begin{proof}
Note in Theorem \ref{thmGT} that we can take $K=[-N,N]^d$ here. For $\psi_i: \mathbb{Z^+}^d\rightarrow \mathbb{Z^+}$, we have
\[
\beta_\infty:=vol_d\left(K\cap\Psi^{-1}((\mathbb{R}^+)^t\r)\ge vol_d\left([-N,N]^d\cap{\mathbb{Z}^+}^d\r)\ge N^d.
\]
Also, the singular product $\prod_{p}\beta_{p}$ would not vanish here since $\beta_p\neq0$ for all $p$. Thus, the asymptotic formula in Theorem \ref{thmGT} owns a dominated main term, and the lemma follows immediately.
\end{proof}

For each $j=1,2,\dots,720$, we apply the coefficients $a_j$, which are defined via
\[
a_j=\frac{720!}{j}.
\]
We consider a system of affine-linear forms $\Psi=\{\psi_1, \dots, \psi_{1440}\}$ via
\begin{align}
&\psi_{2j-1}(n_1,\dots,n_{720},m)=n_j,\notag\\
&\psi_{2j}(n_1,\dots,n_{720},m)=n_j+a_jm\notag
\end{align}
for $j=1,\dots,720$. This is a special system of affine-linear forms with complexity $s=1$, and it is obvious that $\psi_i: \mathbb{Z^+}^d\rightarrow \mathbb{Z^+}$ since all coefficients are positive integers.

To apply Lemma \ref{lem}, we also need $\beta_p\neq0$.
By \eqref{eqbetap} this would be available if, for each $p$, one can find a lattice point $n\in \mathbb{Z}_p^d$ that $(\psi_i(n),p)=1$ holds for all $i$. Obviously, the lattice point $n=(1,1,\cdots,1,0)\in \mathbb{Z}_p^{721}$ is eligible for our system.
Thus, by Lemma \ref{lem}, there are infinitely many lattice points $(n_1,\dots, n_{720}, m)\in {\mathbb{Z}^+}^{721}$, which make all $\psi_i$ prime. That is to say, for each $m$ in these lattice points, the set $\{a_1m, a_2m, \dots, a_{720}m\}$ is consisting of Maillet numbers. We apply $m'$ for the least one of these $m$, which is a computable number since the system of affine-linear forms is specific.

To prove the theorem, we take the constant
\[
k=720!m',
\]
and let the set be
\[
D=\{a_1 m', a_2 m', \dots, a_{720} m'\}.
\]
Also, for any integer $b>0$, Theorem \ref{thmHW} indicates that there is at least a Kronecker number in the set $\{b, 2b, \dots, 720b\}$.
If $jb$ with $1\le j\le 720$ is a Kronecker number, we can write
\[
kb=a_j m'\cdot j b\in D \cdot \mathcal{K},
\]
which establishes Theorem \ref{themain}.

\section{Representation of rationals}

In this section we will use arguments from Ramsey theory to prove Theorem \ref{ratio}. First we will prove three lemmas and then a more general Theorem \ref{abst}. As a corollary we will prove our Theorem \ref{ratio}.

The notion of $IP$ sets and $IP_r$ sets are well studied in Ramsey theory. Let $\mathcal{P}_{f}\left(\mathbb{N}\right)$ be the collection of
nonempty finite subsets of $\mathbb{N}.$
\begin{definition}
 A set $A\subset \mathbb{N}$ is said to be an $IP$ (resp. $IP_r$ for some $r\in \mathbb{N}$) set if there exists a sequence $\langle x_n \rangle_{n\in \mathbb{N}}$ (resp. $\langle x_n \rangle_{n=1}^r$ ) such that $A=FS\left(\langle x_n \rangle_{n\in \mathbb{N}}\right)$ (resp. $A=FS\left(\langle x_n \rangle_{n=1}^r\right)$).
\end{definition}
A set is said to be an $IP^{\star}$ (resp. $IP_r^{\star}$) if this set intersects with every $IP$ set (resp. $IP_r$ sets). Note that every $IP_{r}$ set contains a $\varDelta_{r}$
set. To check this, let $FS\left(\langle x_{n}\rangle_{n=1}^{r}\right)$
be an $IP_{r}$ set and let
\[
S=\left\{ x_{1},x_{1}+x_{2},\ldots,x_{1}+x_{2}+\cdots+x_{n}\right\} ,
\]
Now $FS\left(\langle x_{n}\rangle_{n=1}^{r}\right)$ contains elements
of the form $\left\{ s-t:s>t\text{ and }s,t\in S\right\} $. Hence
every $\varDelta_{r}^{\star}$ set is $IP_{r}^{\star}$. Again every $IP$ set contains an $IP_r$ set for some $r\in \mathbb{N},$ hence every $IP_r^{\star}$ set is $IP^{\star}$. Hence $\mathcal{K}$ is $IP_{721}^{\star}$ and hence an $IP^{\star}$ set. For details on these sets the reader can see the book \cite{HS}.
Let us recall the following theorem of N. Hindman \cite{H}. Before that recall  a sub IP set of  $FS\left(\langle x_{n}\rangle_{n=1}^{\infty}\right)$ is of the form $FS\left(\langle y_{n}\rangle_{n=1}^{\infty}\right)\subseteq FS\left(\langle x_{n}\rangle_{n=1}^{\infty}\right)$, where for each $i\in \mathbb{N},$ $y_i=\sum_{t\in H_i}x_t$ and $H_i\cap H_j = \emptyset$ for $i\neq j.$
\begin{lemma}\label{h} \cite[Hindman theorem]{H}
For every finite partition of an $IP$ set, there exists a partition which contains a sub-IP set.

\end{lemma}

First we will prove three lemmas before we prove Theorem \ref{ratio}. The following lemma says that dilation of any $IP^{\star}$ set by a number is again an $IP^{\star}$ set.
\begin{lemma}\label{ipd}
Let $A\subseteq \mathbb{N},$ be an $IP^{\star}$ set. Then for any $m\in \mathbb{N}$, $m\cdot A=\left\lbrace mx:x\in A\right\rbrace $ is again an $IP^{\star}$ set.
\end{lemma}
\begin{proof}
Let $\langle x_{n}\rangle_{n\in\mathbb{N}}$ be any sequence. For
each
$i\in\mathbb{N}$, let
$$x_{i}\equiv j \mod m$$
where $j\in\left\{ 0,1,\ldots,m-1\right\} .$ Now pick $H_{1}$ (consider a finite
$H_{1}$ to be the collection of those $i$'s such that all $j$'s are same) such that $m\vert\sum_{t\in H_{1}}x_{t}.$ Now
continue this process to obtain a disjoint sequences $\langle H_{n}\rangle_{n\in\mathbb{N}}$
of the finite subsets of $\mathbb{N}$ such that $m\vert\sum_{t\in H_{n}}x_{t}$
for each $n\in\mathbb{N}$. Now choose a new sequence $\langle y_{n}\rangle_{n\in\mathbb{N}}$
such that $y_{n}=\frac{1}{m}\sum_{t\in H_{n}}x_{t}$ for each $n\in\mathbb{N}$.
Then $A\cap FS\left(\langle y_{n}\rangle_{n\in\mathbb{N}}\right)\neq\emptyset$
and this implies $m\cdot A\cap FS\left(\langle x_{n}\rangle_{n\in\mathbb{N}}\right)\neq\emptyset$, finishing the proof.
\end{proof}
The following lemma says that any $IP^{\star}$ set contains an $IP$ set. In fact it contains a sub $IP$ set of any given $IP$ set.
\begin{lemma}\label{ipc}
Let $FS(\langle x_n \rangle_n)$ be any $IP$ set and let $A$ be any $IP^{\star}$ set. Then there exists a sub $IP$ set  $FS(\langle y_n \rangle_n)$ of  $FS(\langle x_n \rangle_n)$ such that $A$ contains $FS(\langle y_n \rangle_n)$.
\end{lemma}

\begin{proof}
 Let us choose the following partition of $FS(\langle x_n \rangle_n)$ by,
 $$FS(\langle x_n \rangle_n)=(A\cap FS(\langle x_n \rangle_n)) \cup ( FS(\langle x_n \rangle_n)\setminus A).$$
Now from Lemma \ref{h}, we have a sub-IP set $FS(\langle y_n \rangle_n))\subseteq FS(\langle x_n \rangle_n))$ such that either $FS(\langle y_n \rangle_n))\subseteq A\cap FS(\langle x_n \rangle_n)$ or $FS(\langle y_n \rangle_n))\subseteq FS(\langle x_n \rangle_n)\setminus A.$ But $A$ is an $IP^{\star}$ set, hence the second case is not possible. So, $FS(\langle y_n \rangle_n))\subseteq A$, and the lemma follows.
\end{proof}
We need the following lemma in our proof. This lemma says that if we take the intersection of finitely many $IP^{\star}$ sets, then still it will be an $IP^{\star}$ set.

\begin{lemma}\label{intip}
Intersection of finitely many $IP^{\star}$ sets is again an $IP^{\star}$ set.
\end{lemma}

\begin{proof}
 Let $n\in \mathbb{N}$. Let $A_1,A_2,\ldots ,A_n$ be $n$ many $IP^{\star}$ sets . Now assume that  $FS(\langle x_n \rangle_n)$ is an $IP$ set. From Lemma \ref{ipc}, we have a sub $IP$ set  $FS(\langle y_n \rangle_n))\subseteq A_1\cap FS(\langle x_n \rangle_n)$ and so $FS(\langle y_n \rangle_n))\subseteq A_1.$ Now again applying Lemma \ref{ipc}, we have $FS(\langle z_n \rangle_n))\subseteq A_2\cap FS(\langle y_n \rangle_n)$ and so $FS(\langle z_n \rangle_n))\subseteq A_2.$ Hence $FS(\langle z_n \rangle_n))\subseteq A_1\cap A_2.$ Iterating this argument we have a sub $IP$ set $FS(\langle a_n \rangle_n)$  of $FS(\langle x_n \rangle_n)$ such that $FS(\langle a_n \rangle_n)\subseteq A_1\cap A_2\cap \ldots \cap A_n.$ As $FS(\langle x_n \rangle_n)$ is any $IP$ set, arbitrarily chosen, we have $A_1\cap A_2\cap \ldots \cap A_n$ is an $IP^{\star}$ set. This accomplishes the proof.
\end{proof}

The following theorem is an abstract formulation of our Theorem \ref{ratio}. In fact this result is much stronger.
\begin{theorem}\label{abst}
If $A$ is any $IP^{\star}$ set and $B$ is an $IP$ set, then $$\mathbb{Q}_{>0}=\frac{A}{B}=\left\lbrace a/b:a\in A, b\in B \right\rbrace.$$
\end{theorem}
\begin{proof}
Let $\frac{m}{n}\in \mathbb{Q}_{>0}.$ Now $m\cdot B$ is an $IP$ set, and from Lemma \ref{ipd}, we have $n\cdot A$ is an $IP^{\star}$ set. Hence, $n\cdot A \cap m\cdot B \neq \emptyset.$ Let $x=na=mb$, where $a\in A, b\in B.$ Then $m/n=a/b\in A/B.$
This completes the proof.
\end{proof}

Now we are ready to prove our Theorem \ref{ratio}.

 \begin{proof} [Proof of Theorem \ref{ratio}]\label{proof} As $\mathcal{K}$ is an $IP_{721}^{\star}$ set, it is an $IP^{\star}$ set. Again using Lemma \ref{ipc}, we can say that $\mathcal{K}$ contains an $IP$ set. Hence from Theorem \ref{abst}, we have our desired result.
 \end{proof}
The above proof of Theorem \ref{ratio} shows a more powerful result, stated below.
\begin{Remark}
For any $IP$ set $\mathcal{D}\subset \mathcal{K}$, we have $\mathbb{Q}_{>0}=\frac{\mathcal{K}}{\mathcal{D}}$, and hence also $\mathbb{Q}_{>0}=\frac{\mathcal{D}}{\mathcal{K}}$.
\end{Remark}


\end{document}